\documentclass[12pt,twoside]{article}
\usepackage{amssymb,amsmath,amsthm,latexsym}
\usepackage{amsfonts}
\usepackage{enumerate}
\usepackage[demo]{graphicx}
\usepackage{caption}
\usepackage{subcaption}
\usepackage{relsize}
\usepackage{tikz}
\usepackage{multirow}
\usepackage{float}
\allowdisplaybreaks
\usepackage{mathrsfs}
\usepackage{doi}
\usetikzlibrary{arrows.meta,positioning, quotes}
\usepackage{calc}

\usetikzlibrary{decorations.markings}
\tikzstyle{vertex}=[circle, draw, inner sep=0pt, minimum size=6pt]
\newcommand{\vertex}{\node[vertex]}
\newtheorem{thm}{Theorem}[section]

\newtheorem{cor}[thm] {Corollary}
\newtheorem{lem} [thm]{Lemma}
\newtheorem{prop} [thm]{Proposition}
\theoremstyle{definition} 

\newtheorem{rmk}[thm] {Remark}
\newtheorem{defn}[thm]{Definition}

\numberwithin{equation}{section}

\newcommand\sgn{\operatorname{sgn}}
\newcommand\bdim{\operatorname{bdim}}

\begin{document}
	\label{'ubf'}
	\setcounter{page}{1} 

	\markboth {\hspace*{-9mm} \centerline{\footnotesize \sc
			Vector Valued Switching in the Products of Signed Graphs}
	}
	{ \centerline {\footnotesize \sc Albin,  Germina}
	}
	\begin{center}
		{
			\Large \textbf{Vector Valued Switching in the Products of Signed Graphs
				}
			}

		\bigskip
					Albin Mathew \footnote{\small Department of Mathematics, Central University of Kerala, Kasaragod - 671316,\ Kerala,\ India. \textbf{Email:}\texttt{ albinmathewamp@gmail.com, }}\, Germina K A \footnote{\small Department of Mathematics, Central University of Kerala, Kasaragod - 671316,\ Kerala,\ India. \textbf{Email:}\texttt{ srgerminaka@gmail.com}}
			\\

\end{center}

\thispagestyle{empty}
\begin{abstract}
A signed graph is a graph whose edges are labeled either as positive or negative. The concept of vector valued switching and balancing dimension of signed graphs were introduced  by S. Hameed et al. In this paper, we deal with the balancing dimension of various products of signed graphs, namely the Cartesian product, the lexicographic product, the tensor product and the strong product.
\end{abstract}




\section{Introduction}

Throughout this paper, unless otherwise mentioned,  we consider only finite, simple, connected and undirected graphs and signed graphs. For the standard notation and terminology in graphs  and signed graphs  not given here, the reader may refer to \cite{fh} and \cite{tz1,tz3} respectively.

  A signed graph $\Sigma=(G,\sigma)$ is a graph $G$, together with a function $\sigma$ that assigns a sign $+1$ or $-1$ to each of its edges. The sign of a cycle in $\Sigma$  is defined as the product of the signs of its edges, and $\Sigma$ is balanced if it does not contain any negative cycles. A signed graph $\Sigma=(G,\sigma)$ is said to be antibalanced if the signed graph $-\Sigma=(G,-\sigma)$ is balanced. A switching function for $\Sigma$ is a function $\zeta:V(\Sigma)\rightarrow\{-1,1\}$. For an edge $e=uv$ in $\Sigma$, the switched signature $\sigma^\zeta$ is defined as $\sigma^\zeta(e)=\zeta(u)\sigma(e)\zeta(v)$,  and the switched signed graph is  $\Sigma^\zeta=(G,\sigma^\zeta)$ . The signs of cycles are unchanged by switching and every balanced (antibalanced) signed graph can be switched to an all-positve (all-negative) signed graph. We call $\Sigma_1$ and $\Sigma_2$ switching equivalent if there is a switching function $\zeta$ such that $\Sigma_2=\Sigma_1^{\zeta}$ (see \cite[Section 3]{tz1}).

  The notion of vector valued switching and balancing dimension of signed graphs were defined by Hameed \textit{et al.} in \cite{sa1}. In this paper, we focus on computing the balancing dimensions of various products of signed graphs such as the Cartesian product, the lexicographic product, the tensor product, and the strong product.

  To begin with, we recall some notations, definitions and fundamental results from \cite{sa1}. In what follows, $\Omega=\{-1,0,1\} $ and the inner product used is the same as that on $\mathbb{R}^k$ restricted to $\Omega^k$.
  \begin{defn}(Vector Valued Switching or $k$-switching) \cite{sa1}
   Let $\Sigma=(G,\sigma)$ be a given signed graph  where $G=(V,E)$.  A vector valued switching function is a function $\zeta:V\rightarrow \Omega^k \subset \mathbb{R}^k$ such that $\langle \zeta(u), \zeta (v)\rangle\neq 0$ for all edges $uv\in E$.  The switched signed graph $\Sigma^\zeta=(G,\sigma^\zeta)$ has the signing $$\sigma^\zeta(uv)=\sigma(uv) \sgn (\langle \zeta(u), \zeta (v)\rangle).$$
  \end{defn}
Note that the switching that has been discussed so far in literature \cite{tz1} can be considered as $1$-switching.
Using  vector valued switching, the balancing dimension for a signed graph is defined as follows.
\begin{defn} (Balancing Dimension) \cite{sa1}
	Let $\Sigma=(G,\sigma)$ be a given signed graph  where $G=(V,E)$. We say that the balancing dimension of $\Sigma$ is $k$ and write it as $\bdim(\Sigma)$, if $k\ge 1$ is the least integer such that a vector valued switching function $\zeta:V\rightarrow \Omega^k\subset \mathbb{R}^k$ switches $\Sigma$ to an all positive signed graph.

	Such a $k$-switching function  $\zeta$  is called a positive $k$-switching function (briefly a $k$-positive function) for $\Sigma$.
\end{defn}
One may note that $\bdim(\Sigma)=1$ if and only if $\Sigma$ is balanced. Also, the balancing dimension of a subgraph of $\Sigma$ cannot exceed the balancing dimension of $\Sigma$. We will also make use of the fact that balancing dimension is $1$-switching  invariant (see \cite{sa1}).
\section{Balancing dimension of the product of signed graphs}
In this section, we establish some results regarding  the balancing dimension of the Cartesian product, the lexicographic product, the tensor product, and the strong product  of signed graphs.

\subsection{Balancing dimension of the Cartesian product}
 The Cartesian product of two signed graphs is defined by Germina et al. in \cite{sg1}.

\begin{defn}\cite{sg1}
	The Cartesian product $\Sigma_1\Box \Sigma_2$ of two signed graphs $\Sigma_1=(G_1,\sigma_1)$ and $\Sigma_2=(G_2,\sigma_2)$ is defined as the Cartesian product of the underlying unsigned graphs with the signature function $\sigma$ for the labeling of the edges defined by
	$$\sigma((u_i,v_j)(u_k,v_l))=
	\begin{cases}
		\sigma_1(u_iu_k), \mbox{ if } j=l,\\
		\sigma_2(v_jv_l),  \mbox{ if } i=k.
	\end{cases}
	$$
\end{defn}
If $\Sigma_1$ and $\Sigma_2$ are balanced, then their Cartesian product $\Sigma_1\Box\Sigma_2$ is also balanced (see \cite{sg1}) and hence $\bdim(\Sigma_1\Box\Sigma_2)=1$. We now consider the case where one of the factors is balanced.
\begin{thm}
Let $\Sigma_1=(G_1,\sigma_1)$ and $\Sigma_2=(G_2,\sigma_2)$ be two signed graphs and let $\Sigma_1\Box\Sigma_2$ be their Cartesian product. Then
	$$\bdim(\Sigma_1\Box\Sigma_2)=
\begin{cases}
	\bdim(\Sigma_1), \mbox{  if $\Sigma_2$   is balanced}\\
	\bdim(\Sigma_2), \mbox{  if $\Sigma_1$   is balanced}.
\end{cases}
$$
\end{thm}
\begin{proof}
	Suppose $\bdim(\Sigma_1)=k$ and $\Sigma_2$ is balanced. Let $\zeta_1:V(\Sigma_1)\rightarrow\Omega^k$ and $\zeta_2:V(\Sigma_2)\rightarrow \{-1,1\}$ be the corresponding switching functions. We now define $\zeta: V(\Sigma_1\times\Sigma_2)\rightarrow\Omega^k$ by $\zeta((u_i,v_j))=\zeta_1(u_i)\zeta_2(v_j)$ for $1\leq i \leq |V(\Sigma_1)|$ and $1\leq j \leq |V(\Sigma_2)|$.

	Now for any edge $e=(u_i,v_j)(u_k,v_l)$ in $\Sigma_1\Box\Sigma_2$, we have,
	\begin{equation}\label{eqn1}
		\sigma^\zeta((u_i,v_j)(u_k,v_l))=\sigma((u_i,v_j)(u_k,v_l))\sgn(\langle\zeta((u_i,v_j)),\zeta((u_k,v_l))\rangle).
	\end{equation}
	If $j=l$, Equation~\ref*{eqn1} becomes
		\begin{align*}
		\sigma^\zeta((u_i,v_l)(u_k,v_j))
		&=\sigma_1(u_iu_k)\sgn(\langle\zeta((u_i,v_j)),\zeta((u_k,v_l))\rangle)\\
		&=\sigma_1(u_iu_k)\sgn(\langle\zeta_1(u_i)\zeta_2(v_l),\zeta_1(u_k)\zeta_2(v_l)\rangle)\\
		&=(\zeta_2(v_l))^2\sigma_1(u_iu_k)\sgn(\langle\zeta_1(u_i),\zeta_1(u_k)\rangle)\\
		&=(\zeta_2(v_l))^2\sigma_1^{\zeta_1}(u_iu_k)\\
		&=+1.
		\end{align*}
	Similarly, if $i=k$, Equation~\ref*{eqn1} becomes
		\begin{align*}
			\sigma^\zeta((u_i,v_j)(u_k,v_l))
			&=\sigma_2(v_jv_l)\sgn(\langle\zeta((u_k,v_j)),\zeta((u_k,v_l))\rangle)\\
			&=\sigma_2(v_jv_l)\sgn(\langle\zeta_1(u_k)\zeta_2(v_j),\zeta_1(u_k)\zeta_2(v_l)\rangle)\\
			&=\sigma_2(v_jv_l)\zeta_2(v_j)\zeta_2(v_l)\sgn(\langle\zeta_1(u_k),\zeta_1(u_k)\rangle)\\
			&=\sigma_2^{\zeta_2}(v_jv_l)\sgn(\lVert\zeta_1(u_k)\rVert^2)\\
			&=+1.
		  	\end{align*}

Thus, $\zeta$ switches $\Sigma_1\Box\Sigma_2$ to all-positive, and hence $\bdim(\Sigma_1\Box\Sigma_2)\leq k$. However, since $\Sigma_1$ is a subgraph of $\Sigma_1\Box\Sigma_2$, we must have $\bdim(\Sigma_1\Box\Sigma_2)=k=\bdim(\Sigma_1)$.

Similar is the proof of the next part.
\end{proof}

 \begin{thm}\cite{dl}
 	Let $\Sigma_1$ and $\Sigma_2$ be two signed graphs and let $\Sigma_1\Box\Sigma_2$ be their Cartesian product. If $\Sigma_1 \simeq \Sigma_1'$ and $\Sigma_2 \simeq \Sigma_2'$, then $\Sigma_1\Box\Sigma_2\simeq\Sigma_1'\Box\Sigma_2'$.
 \end{thm}
 \begin{thm}\cite{sa1}\label{negtriangle}
 	If $\Sigma$ contains a negative triangle, then $\bdim(\Sigma)\geq3$.
 \end{thm}

  We now compute the balancing dimension of the Cartesian product of two unbalanced signed graphs. To begin with, we consider the Cartesian product of two unbalanced cycles.

 \begin{prop}
 	Let $C_m^-$ and $C_n^-$, $m,n\geq3$ be two unbalanced cycles and let $C_m^- \Box C_n^-$ be their Cartesian product. Then,
 		$$\bdim(C_m^-\Box C_n^-)=
 	\begin{cases}
 		2, \mbox{  if $m,n>3$}\\
 		3, \mbox{ otherwise}.
 	\end{cases}$$
\end{prop}
 \begin{proof}
 Since balancing dimension is $1$ - switching invariant, by using Theorem~\cite{dl}, we can consider $C_m^-=u_1u_2\cdots u_m$ and $C_n^-=v_1v_2\cdots v_n$, where $u_1u_2$ and $v_1v_2$ are the only negative edges of $C_m^-$ and $C_n^-$ respectively. We need to consider four cases: $n,m>3$, $m=3$ and $n>3$, $m>3$ and $n=3$, and $n=m=3$.

 In the first case, since $C_m^- \Box C_n^-$ is unbalanced, we have $\bdim(C_m^-\Box C_n^-)\geq 2$. Now, the function $\zeta_1:V(C_m^-\Box C_n^-)\rightarrow\Omega^2$ given in Table~\ref{table:1} switches $C_m^-\Box C_n^-$ to all-positive. Hence, $\bdim(C_m^-\Box C_n^-)=2$.
 In the next three cases, $C_m^- \Box C_n^-$ contains a negative triangle and hence  $\bdim(C_m^-\Box C_n^-)\geq 3$ . Now, the functions $\zeta_i:V(C_m^-\Box C_n^-)\rightarrow\Omega^3$, where $i\in\{2,3,4\}$ given in Tables~\ref{table:2} to \ref{table:4} respectively, switches $C_m^-\Box C_n^-$ to all-positive. Hence $\bdim(C_m^-\Box C_n^-)=3$ in each of these cases.
 \end{proof}
 \begin{table}[H]
 	\centering
 \begin{tabular}{|c|c|c|c|c|}
 	\hline
 	$\zeta_1((u_i,v_j))$&$v_1$  & $v_2$ & $v_3,v_4, \cdots , v_{n-1}$  & $v_n$ \\
 	\hline
 	$u_1$& $(-1,1)$ &$(1,0)$  & $(1,1)$  & $(0,1)$ \\
 	\hline
 	$u_2$& $(1,0)$ & $(-1,-1)$  & $(0,-1)$ & $(1,-1)$ \\
 	\hline
 	$u_3,u_4, \cdots , u_{m-1}$ & $(1,1)$  & $(0,-1)$  & $(1,-1)$ & $(1,0)$  \\
 	\hline
 	$u_m$& $(0,1)$ & $(1,-1)$  & $(1,0)$  & $(1,1)$  \\
 	\hline
 \end{tabular}
 \caption{A $2$ - positive function for $C_m^-\Box C_n^-$ for $n,m>3$}
 \label{table:1}
 \end{table}
\begin{table}[H]
	\centering
	\begin{tabular}{|c|c|c|c|}
		\hline
		$\zeta_2((u_i,v_j))$& $v_1$ & $v_2,v_3, \cdots , v_{n-1}$  & $v_n$  \\
		\hline
		$u_1$& $(-1,-1,1)$ & $(1,-1,-1)$  &  $(-1,-1,-1)$\\
		\hline
		$u_2$&  $(1,-1,-1)$ &  $(1,1,1)$& $(1,0,0)$ \\
		\hline
		$u_3$& $(-1,-1,-1)$ & $(1,0,0)$ & $(1,-1,-1)$  \\
		\hline
	\end{tabular}
	 \caption{A $3$ - positive function for $C_3^-\Box C_n^-$ for $n>3$}
	\label{table:2}
\end{table}
\begin{table}[H]
	\centering
	\begin{tabular}{|c|c|c|c|}
	\hline
	$\zeta_3((u_i,v_j))$& $v_1$ & $v_2$ & $v_3$ \\
	\hline
	$u_1$& $(-1,-1,1)$ & $(1,-1,-1)$ & $(-1,-1,-1)$   \\
	\hline
	$u_2,u_3, \cdots , u_{m-1}$& $(1,-1,-1)$ &  $(1,1,1)$& $(1,0,0)$ \\
	\hline
	$u_m$& $(-1,-1,-1)$ &  $(1,0,0)$ & $(1,-1,-1)$ \\
	\hline
	\end{tabular}
	\caption{A $3$ - positive function for $C_m^-\Box C_3^-$ for $m>3$}
	\label{table:3}
\end{table}
\begin{table}[H]
	\centering
	\begin{tabular}{|c|c|c|c|}
		\hline
		$\zeta_4((u_i,v_j))$& $v_1$ & $v_2$ & $v_3$ \\
		\hline
		$u_1$& $(-1,-1,1)$ & $(1,-1,-1)$ & $(-1,-1,-1)$   \\
		\hline
		$u_2$& $(1,-1,-1)$ &  $(1,1,1)$& $(1,0,0)$ \\
		\hline
		$u_3$& $(-1,-1,-1)$ &  $(1,0,0)$ & $(1,-1,-1)$ \\
		\hline
	\end{tabular}
	\caption{A $3$ - positive function for $C_3^-\Box C_3^-$}
	\label{table:4}
\end{table}
Next, we consider antibalanced signed complete graphs. We denote the antibalanced signed complete graph on $n$ vertices by $K_n^-$, and the balancing dimension of $K_n^-$ is defined as $\bar{\nu}(n)$ \cite{sa1}.
 \begin{prop}
	Let $K_m^-$ and $K_n^-$ be antibalanced signed complete graphs of order $m$ and $n$ respectively,   and let $K_m^- \Box K_n^-$ be their Cartesian product. Then
	$$\bdim(K_m^-\Box K_n^-)=
	\begin{cases}
		\bar\nu(m), \mbox{  if $m\geq n$}\\
		\bar\nu(n), \mbox{  if $m\leq n$}.
	\end{cases}$$
	where, $\bar\nu(m)=\bdim(K_m^-)$.
\end{prop}
\begin{proof}
By adequate $1$-switching, we can consider  $K_m^-$ and $K_n^-$ as all-negative. Then, the Cartesian product $K_m^-\Box K_n^-$ is also all-negative.

Without loss of generality, assume that $m\geq n$. Suppose $\bdim(K_m^-)=k$ and let $\zeta':V(K_m^-)\rightarrow\Omega^k$ be the $k$ - positive function. Since $K_m^-$ is a subgraph of $K_m^-\Box K_n^-$, we have  $\bdim(K_m^-\Box K_n^-)\geq k$. Now, the function $\zeta:V(K_m^-\Box K_n^-)\rightarrow\Omega^k$ given in Table~\ref{table:5} switches $K_m^-\Box K_n^-$ to all-positive. Hence, $\bdim(K_m^-\Box K_n^-)=k=\bar{\nu}(m)$.

Similar is the proof of the next part.
\end{proof}
\begin{table}[H]
	\centering\begin{tabular}{|c|c|c|c|c|c|}
		\hline
		$\zeta((u_i,v_j))$& $v_1$ & $v_2$ & $v_3$  & $\cdots$ & $v_n$ \\
		\hline
		$u_1$&$\zeta_1(u_1)$  & $\zeta_1(u_2)$ & $\zeta_1(u_3)$ & $\cdots$ & $\zeta_1(u_n)$ \\
		\hline
		$u_2$& $\zeta_1(u_2)$ & $\zeta_1(u_3)$ & $\zeta_1(u_4)$ & $\cdots$ & $\zeta_1(u_{n+1})$ \\
		\hline
		$u_3$& $\zeta_1(u_3)$ & $\zeta_1(u_4)$ & $\zeta_1(u_5)$ & $\cdots$ &  $\zeta_1(u_{n+2})$\\
		\hline
		\vdots & \vdots & \vdots & \vdots & \vdots \vdots \vdots & \vdots \\
		\hline
		$u_m$& $\zeta_1(u_m)$ & $\zeta_1(u_1)$ & $\zeta_1(u_2)$ & $\cdots$  & $\zeta_1(u_{n-1})$ \\
		\hline
	\end{tabular}
	\caption{A $k$ - positive function for $K_m^-\Box K_n^-$}
	\label{table:5}
\end{table}
\begin{cor}
	For any antibalanced signed graph $\Sigma$ on $n$ vertices, $\bdim(\Sigma \Box K_n^-)=\bdim(K_n^-)$.
\end{cor}

%
\subsection{Balancing dimension of the lexicographic product}
We now focus on the lexicographic product (also called composition) of signed graphs. Two definitions for the lexicographic product  of signed graphs are available in the literature. We call the definition given by Hameed et al.~\cite{sg2} the \textit{HG-lexicographic product} and the definition given by Brunetti et al.~\cite{bcd} the \textit{BCD-lexicographic product}.
\begin{defn}\cite{sg2}
	The \textit{HG-lexicographic product} of two signed graphs $\Sigma_1=(G_1,\sigma_1)$ and $\Sigma_2=(G_2,\sigma_2)$ is the signed graph  whose underlying graph is the lexicographic product of underlying unsigned graphs and whose signature function $\sigma$ for the labeling of the edges is defined by
	$$\sigma((u_i,v_j)(u_k,v_l))=
	\begin{cases}
		\sigma_1(u_iu_k)  & \mbox{if } i\neq k, \\
		\sigma_2(v_jv_l) & \mbox{if } i=k.
	\end{cases}
	$$

	We denote the \textit{HG-lexicographic product} of $\Sigma_1$ and $\Sigma_2$ by $\Sigma_1[\Sigma_2]$.
\end{defn}

\begin{defn}\cite{bcd}
	The \textit{BCD-lexicographic product} of two signed graphs $\Sigma_1=(G_1,\sigma_1)$ and $\Sigma_2=(G_2,\sigma_2)$ as the signed graph whose underlying graph is the lexicographic product of underlying unsigned graphs and whose signature function $\sigma$ for the labeling of the edges is defined by
	$$\sigma((u_i,v_j)(u_k,v_l))=
	\left\{
	\begin{array}{ll}
		\sigma_1(u_iu_k)  & \mbox{if } u_i \sim u_k\  \mbox{and } v_j\nsim v_l, \\
		\sigma_1(u_iu_k)\sigma_2(v_jv_l)  & \mbox{if } u_i \sim u_k\  \mbox{and } v_j\sim v_l, \\
		\sigma_2(v_jv_l)  & \mbox{if } u_i=u_k\  \mbox{and } v_j\sim v_l. \\
	\end{array}
	\right.$$

	We denote the \textit{BCD-lexicographic product} of $\Sigma_1$ and $\Sigma_2$ by $\Sigma_1*\Sigma_2$.
\end{defn}
The \textit{HG-lexicographic product} and \textit{BCD-lexicographic product}  of two balanced signed graphs need not be balanced. However, a criterion for the balance of \textit{HG-lexicographic product} of two signed graphs is proved in  \cite{sg2}
.
\begin{thm}\cite{sg2}\label{lexchar}
	If $\Sigma_1$ and $\Sigma_2$ are two signed graphs with at least one edge for each, then their HG-lexicographic product  is balanced if and only if $\Sigma_1$ is balanced and $\Sigma_2$ is all-positive.
\end{thm}

\begin{thm}\label{hggeq3}
	Let $\Sigma_1$ and $\Sigma_2$ are two signed graphs, with $\Sigma_2$ having at least one negative edge. Then $\bdim(\Sigma_1[\Sigma_2]) \geq 3$.
\end{thm}
\begin{proof}
	Let $v_jv_{j+1}$ be a negative edge of $\Sigma_2$ Then for any edge $u_iu_{i+1}$ of $\Sigma_1$,  $(u_i,v_j)(u_i,v_{j+1})(u_{i+1},v_j)(u_i,v_j)$ forms a negative triangle in $\Sigma_1[\Sigma_2]$. Hence, by Theorem~\ref{negtriangle}, $\bdim(\Sigma_1[\Sigma_2]) \geq 3$.
\end{proof}
\begin{prop}\label{hgnk}
	For any signed graph $\Sigma$, $\bdim(N_k[\Sigma])=\bdim(\Sigma[N_k])=\bdim(\Sigma)$, where $N_k$ is the graph on $k$ vertices without edges.
\end{prop}
\begin{proof}
	Suppose $\bdim(\Sigma)=k$ and $\zeta:V(\Sigma)\rightarrow\Omega^k$ be the $k$-positive function. Then,  $\zeta':V(N_k[\Sigma])\rightarrow\Omega^k$ defined by $\zeta'((w_i,u_j))=\zeta(u_j)$ for $1\leq i \leq k$, $1\leq j \leq |V(\Sigma)|$ switches $N_k[\Sigma]$ to all-positive. Hence,  $\bdim(N_k[\Sigma])\leq k$. However, since $\Sigma$ is a subgraph of $N_k[\Sigma]$, we must have $\bdim(N_k[\Sigma])=k=\bdim(\Sigma)$.

	Similarly, $\zeta'':V(\Sigma[N_k])\rightarrow\Omega^k$ defined by $\zeta''((u_i,w_j))=\zeta(u_i)$ for $1\leq i \leq |V(\Sigma)|$, $1\leq j \leq k$ switches $\Sigma[N_k]$ to all-positive, and hence $\bdim(\Sigma[N_k])=k$.
\end{proof}
	\begin{rmk}
		The above results shows that, even though the lexicographic product is not commutative, there exists signed graphs satisfying $\bdim(\Sigma_1[\Sigma_2])=\bdim(\Sigma_2[\Sigma_1])$. However, in general, $\bdim(\Sigma_1[\Sigma_2])\neq\bdim(\Sigma_2[\Sigma_1])$. As an example, consider $\Sigma_1$ as the balanced triangle having two negative edges and $\Sigma_2$ as the positive $K_2$. Then, Theorem~\ref{lexchar} and Theorem~\ref{hggeq3} respectively shows that  $\bdim(\Sigma_1[\Sigma_2])=1$ and $\bdim(\Sigma_2[\Sigma_1])\geq 3$.
	\end{rmk}

\begin{thm}\label{hglexswitcheq}
	Let $\Sigma_1$ and $\Sigma_2$ be two signed graphs and let $\Sigma_1[\Sigma_2]$ be their HG - lexicographic product. If $\Sigma_1 \sim \Sigma_1'$, then $\Sigma_1[\Sigma_2]\sim\Sigma_1'[\Sigma_2]$.
\end{thm}
\begin{proof}
	Let $\sigma_1$, $\sigma_1'$, $\sigma_2$, $\sigma$ and $\sigma'$ denote the signatures of $\Sigma_1$, $\Sigma_1'$, $\Sigma_2$, $\Sigma_1[\Sigma_2]$ and $\Sigma_1'[\Sigma_2]$ respectively. Since $\Sigma_1 \sim \Sigma_1'$, there exist a switching function $\eta:V(\Sigma_1)\rightarrow\{-1,1\}$ such that $\Sigma_1^\eta=\Sigma_1'$. Define the map $\eta':V(\Sigma_1[\Sigma_2])\rightarrow\{-1,1\}$ as $\eta'((u_i,v_j))=\eta(u_i)$ for $1\leq i \leq |V(\Sigma_1)|$ and $1\leq j \leq |V(\Sigma_2)|$. Then, for any edge $(u_i,v_j)(u_k,v_l)$ in $\Sigma_1[\Sigma_2]$, we have,
	\begin{align*}
		\sigma^{\eta'}((u_i,v_j)(u_k,v_l))
		&=\eta'((u_i,v_j))\sigma((u_i,v_j)(u_k,v_l))\eta'((u_k,v_l))\\
		&=\eta(u_i)\sigma((u_i,v_j)(u_k,v_l))\eta(u_k)\\
		&=\begin{cases}
				\eta(u_i)\sigma_1(u_iu_k)\eta(u_k)  & \mbox{if } i\neq k, \\
			\eta(u_k)\sigma_2(v_jv_l)\eta(u_k) & \mbox{if } i=k.
		\end{cases}\\
		&=\begin{cases}
			\sigma_1^\eta(u_iu_k)  & \mbox{if } i\neq k, \\
			\sigma_2(v_jv_l) & \mbox{if } i=k.
		\end{cases}\\
		&=\begin{cases}
			\sigma_1'(u_iu_k)  & \mbox{if } i\neq k, \\
			\sigma_2(v_jv_l) & \mbox{if } i=k.
		\end{cases}\\
		&=	\sigma'((u_i,v_j)(u_k,v_l)).
	\end{align*}
Thus,$(\Sigma_1[\Sigma_2])^{\eta'}=\Sigma_1'[\Sigma_2]$ and hence, $\Sigma_1[\Sigma_2] \sim \Sigma_1'[\Sigma_2]$.
\end{proof}

Since balancing dimension is $1$-switching  invariant, we have the following result.
\begin{cor}\label{hgc1}
	If $\Sigma_1$ and $\Sigma_2$ are any two signed graphs and if $\Sigma_1\sim\Sigma_1'$, then $\bdim(\Sigma_1[\Sigma_2])=\bdim(\Sigma_1'[\Sigma_2])$.
\end{cor}
\begin{cor}\label{hgc2}
	If $\Sigma_1$ is antibalanced and $\Sigma_2$ is all-negative, then $\Sigma_1[\Sigma_2]$ is antibalanced.
\end{cor}
\begin{proof}
Since $\Sigma_1$ is antibalanced, we have $\Sigma_1\sim\Sigma_1'$, where $\Sigma_1'$ is all- negative. Now, since $\Sigma_2$ is all-negative, $\Sigma_1'[\Sigma_2]$ is all-negative, and hence antibalanced. Thus, by Theorem~\ref{hglexswitcheq}, $\Sigma_1[\Sigma_2]$ is antibalanced .
\end{proof}
We now consider complete graphs. To begin with, observe that the lexicographic product of two complete graphs, say $K_m$ and $K_n$ is the complete graph $K_{mn}$. To see this, consider any two vertices $u=(u_i,v_j)$ and $v=(u_k,v_l)$ in $K_m[K_n]$. Then $u_i, u_k$ are adjacent in $K_m$ and $v_k, v_l$ are adjacent in $K_n$. Now if $u_i\neq u_k$, then  $u_i$ and $u_k$ are adjacent in $K_m$. Therefore, $u$ and $v$ are adjacent in $K_m[K_n]$. On the other hand, if $u_i = u_k$, then  since $v_k, v_l$ are adjacent in $K_n$, $u$ and $v$ are adjacent in $K_m[K_n]$. Thus, any two of the $mn$ vertices of $K_m[K_n]$ are adjacent.

The next result follows immediately from Corollary~\ref{hgc2}
\begin{prop}
	If  $\Sigma_1$ and $\Sigma_2$ denote the antibalanced signed complete graph $K_m^-$ and the all-negative signed complete graph $-K_n$ respectively, then $\bdim(\Sigma_1[\Sigma_2])=\bar\nu(mn)$, where $\bar\nu(mn)$ is the balancing dimension of the antibalanced signed complete graph $K_{mn}^-$.
\end{prop}
\begin{thm}
	Let $\Sigma_1=(G_1,\sigma_1)$ and $\Sigma_2=(G_2,\sigma_2)$ be two signed graphs, where $\Sigma_2$ is all-positive. Then $\bdim(\Sigma_1[\Sigma_2])=\bdim(\Sigma_1)$.
\end{thm}
\begin{proof}
	Suppose $\Sigma_2$ is all positive. Let $\bdim(\Sigma_1)=k$ and $\zeta_1:V(\Sigma_1)\rightarrow\Omega^k$ be the $k$-positive function. Now, the function $\zeta:V(\Sigma_1[\Sigma_2])\rightarrow\Omega^k$ defined by  $\zeta((u_i,v_j))=\zeta_1(u_i)$ for $1\leq i \leq |V(\Sigma_1)|$ and $1\leq j \leq |V(\Sigma_2)|$
 switches $\Sigma_1[\Sigma_2]$ to all-positive, and hence $\bdim(\Sigma_1[\Sigma_2])\leq k$.  However, since $\Sigma_1$ is a subgraph of $\Sigma_1[\Sigma_2]$, we must have $\bdim(\Sigma_1[\Sigma_2])=k=\bdim(\Sigma_1)$.
\end{proof}
\begin{rmk}
	Let $\Sigma_1=+K_2$ and $\Sigma_2=-K_2$. Then $\Sigma_1[\Sigma_2]$ is the antibalanced signed complete graph $K_4^-$ and hence $\bdim(\Sigma_1[\Sigma_2])=\bar{\nu}(4)=3\neq\bdim(\Sigma_2)$. Thus, $\bdim(\Sigma_1[\Sigma_2])$ and $\bdim(\Sigma_2)$ need not be equal if $\Sigma_1$ is all-positive.
\end{rmk}
We now focus on the \textit{BCD-lexicographic product} of two signed graphs. To begin with, we restate  the Theorem 2.3 from \cite{bcd}, by removing the incorrect part (see \cite{bcd1}) and provide an alternate proof for it.

\begin{thm}\label{bcdlexswitcheq}
	Let $\Sigma_1$ and $\Sigma_2$ be two signed graphs and let $\Sigma_1*\Sigma_2$ be their BCD - lexicographic product. If $\Sigma_1 \sim \Sigma_1'$, then $\Sigma_1*\Sigma_2\sim\Sigma_1'*\Sigma_2$.
\end{thm}
\begin{proof}
	Let $\sigma_1$, $\sigma_1'$, $\sigma_2$, $\sigma$ and $\sigma'$ denote the signatures of $\Sigma_1$, $\Sigma_1'$, $\Sigma_2$, $\Sigma_1*\Sigma_2$ and $\Sigma_1'*\Sigma_2$ respectively. Since $\Sigma_1 \sim \Sigma_1'$, there exist a switching function $\eta:V(\Sigma_1)\rightarrow\{-1,1\}$ such that $\Sigma_1^\eta=\Sigma_1'$. Define the map $\eta':V(\Sigma_1*\Sigma_2)\rightarrow\{-1,1\}$ as $\eta'((u_i,v_j))=\eta(u_i)$ for $1\leq i \leq |V(\Sigma_1)|$ and $1\leq j \leq |V(\Sigma_2)|$. Then, for any edge $(u_i,v_j)(u_k,v_l)$ in $\Sigma_1*\Sigma_2$, we have, $\sigma^{\eta'}((u_i,v_j)(u_k,v_l))=\sigma'((u_i,v_j)(u_k,v_l))$. Thus, $(\Sigma_1*\Sigma_2)^{\eta'}=\Sigma_1'*\Sigma_2$
and hence $\Sigma_1*\Sigma_2\sim \Sigma_1'*\Sigma_2$.
\end{proof}

Since balancing dimension is $1$-switching  invariant, we have the following result.
\begin{cor}
	If $\Sigma_1$ and $\Sigma_2$ are any two signed graphs and if $\Sigma_1\sim\Sigma_1'$, then $\bdim(\Sigma_1*\Sigma_2)=\bdim(\Sigma_1'*\Sigma_2)$.
\end{cor}
\begin{rmk}
	Note that Corollary~\ref{hgc2} does not hold  in the case of \textit{BCD-lexicographic product}. To illustrate this consider $\Sigma_1=(P_3,\sigma)$ and $\Sigma_2=-P_2$ depicted in Figure~\ref{fig2}. Then, $(u_1,v_1)(u_2,v_1)(u_2,v_2)(u_1,v_1)$ forms a negative triangle in $-(\Sigma_1*\Sigma_2)$, making it unbalanced. Thus, $\Sigma_1*\Sigma_2$ is not antibalanced, though $\Sigma_1$ is antibalanced and $\Sigma_2$ is all-negative.

		\begin{figure}[h!]
		\centering
		\begin{subfigure}{.25\textwidth}
			\centering
			\begin{tikzpicture}[x=0.5cm, y=0.5cm]
				\vertex[fill] (v1) at (0,0) [label=below:$u_1$] {};
				\vertex[fill] (v2) at (3,0) [label=below:$u_2$] {};
				\vertex[fill] (v3) at (6,0) [label=below:$u_3$] {};
				\path[very thick, dotted]
				(v1) edge (v2)
				;
				\path
				(v2) edge (v3)
				;
			\end{tikzpicture}
			\caption{$\Sigma_1=(P_3,\sigma)$}
			\label{fig:sub3}
		\end{subfigure}
				\begin{subfigure}{.25\textwidth}
			\centering
			\begin{tikzpicture}[x=0.5cm, y=0.5cm]
				\vertex[fill] (v1) at (0,0) [label=below:$v_1$] {};
				\vertex[fill] (v2) at (3,0) [label=below:$v_2$] {};
				\path[very thick, dotted]
				(v1) edge (v2)
				;
			\end{tikzpicture}
			\caption{$\Sigma_2=-K_2$}
			\label{fig:sub4}
		\end{subfigure}
				\begin{subfigure}{.25\textwidth}
			\centering
	\begin{tikzpicture}[x=0.6cm, y=0.6cm]
	\draw[fill=black](0,0)circle(3pt);
	\draw[fill=black](3,0)circle (3pt);
	\draw[fill=black](0,-3)circle (3pt);
	\draw[fill=black](3,-3)circle (3pt);
	\draw[fill=black](0,-6)circle(3pt);
	\draw[fill=black](3,-6)circle (3pt);
	\node at(-1.5,.75){$(u_1,v_1)$};
	\node at(4.5,.75){$(u_1,v_2)$};
	\node at(-1.5,-3){$(u_2,v_1)$};
	\node at(4.5,-3){$(u_2,v_2)$};
	\node at(-1.5,-6.75){$(u_3,v_1)$};
	\node at(4.5,-6.75){$(u_3,v_2)$};
	\draw[thick](0,0)--(0,-3)--(3,-6);
	\draw[thick](3,0)--(3,-3)--(0,-6);
	\draw[dashed] (0,0)--(3,0)--(0,-3)--(3,-3)--(0,0);
	\draw[dashed] (0,-3)--(0,-6)--(3,-6)--(3,-3);
\end{tikzpicture}
			\caption{$\Sigma_1*\Sigma_2$}
			\label{fig:sub5}
		\end{subfigure}

\caption{The \textit{BCD-lexicographic product} $\Sigma_1*\Sigma_2$ is not antibalanced}
\label{fig2}
\end{figure}
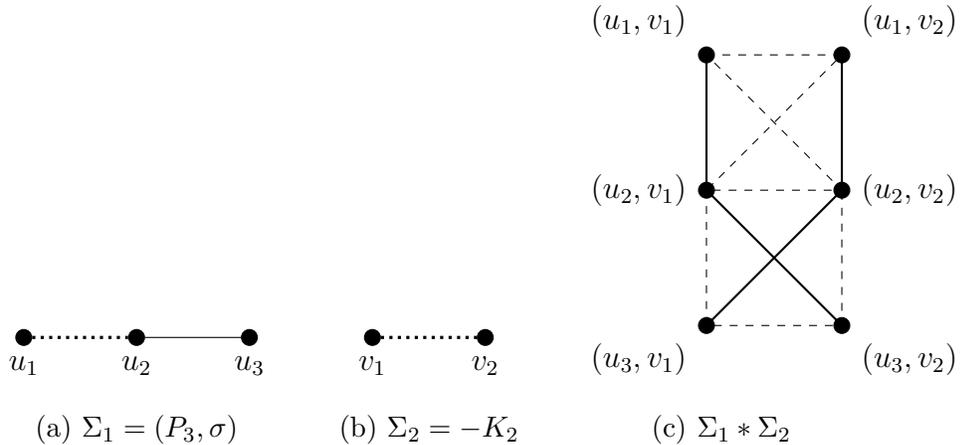

\end{rmk}
\begin{thm}
	Let $\Sigma_1=(G_1,\sigma_1)$ and $\Sigma_2=(G_2,\sigma_2)$ be two signed graphs, where $\Sigma_2$ is all-positive. Then $\bdim(\Sigma_1*\Sigma_2)=\bdim(\Sigma_1)$.
\end{thm}
\begin{proof}
	Suppose $\Sigma_2$ is all positive and let $\bdim(\Sigma_1)=k$. Then the function $\zeta:V(\Sigma_1*\Sigma_2)\rightarrow\Omega^k$ defined by  $\zeta((u_i,v_j))=\zeta_1(u_i)$ for $1\leq i \leq |V(\Sigma_1)|$ and $1\leq j \leq |V(\Sigma_2)|$, where $\zeta_1:V(\Sigma_1)\rightarrow\Omega^k$ is the $k$-positive function for $\Sigma_1$,  switches $\Sigma_1*\Sigma_2$ to all-positive, and hence $\bdim(\Sigma_1*\Sigma_2)\leq k$.  However, since $\Sigma_1$ is a subgraph of $\Sigma_1*\Sigma_2$, we must have $\bdim(\Sigma_1*\Sigma_2)=k=\bdim(\Sigma_1)$.
\end{proof}

\begin{thm}\label{bcdthm3}
	Let $\Sigma_1=(G_1,\sigma_1)$ be a balanced signed graph and  $\Sigma_2=(K_n,\sigma_2)$ be a signed complete graph. Then $\bdim(\Sigma_1*\Sigma_2)=\bdim(\Sigma_2)$.
\end{thm}
\begin{proof}
	Since, $\Sigma_1$ is balanced, by Theorem~\ref{bcdlexswitcheq}, we can consider it as all-positive. Let $\bdim(\Sigma_2)=k$ and let $\zeta_2:V(\Sigma_2)\rightarrow\Omega^k$ be the corresponding $k$-positive function. Then the vector valued switching function $\zeta:V(\Sigma_1*\Sigma_2)\rightarrow\Omega^k$ defined by  $\zeta((u_i,v_j))=\zeta_2(v_j)$ for $1\leq i \leq |V(\Sigma_1)|$ and $1\leq j \leq n$ switches $\Sigma_1*\Sigma_2$ to all-positive, and hence $\bdim(\Sigma_1*\Sigma_2)\leq k$.  However, since $\Sigma_2$ is a subgraph of $\Sigma_1*\Sigma_2$, we must have $\bdim(\Sigma_1*\Sigma_2)=k=\bdim(\Sigma_2)$.
\end{proof}
\begin{rmk}
	The Theorem~\ref{bcdthm3} does not hold for the \textit{HG-lexicographic product} of signed graphs. As an example, let $\Sigma_1=+K_2$ and $\Sigma_2=-K_2$. Then the \textit{HG-lexicographic product} $\Sigma_1[\Sigma_2]$ is the antibalanced signed complete graph $(K_4,\sigma)$ and hence $\bdim(\Sigma_1[\Sigma_2])=\bar{\nu}(4)=3\neq\bdim(\Sigma_2)$.
	Thus, $\bdim(\Sigma_1[\Sigma_2])$ and $\bdim(\Sigma_2)$ need not be equal if $\Sigma_1$ is balanced and $\Sigma_2$ is a signed complete graph.
\end{rmk}
\subsection{Balancing dimension of the tensor product}
We now focus on the tensor product of signed graphs. The tensor product of two signed graphs is given in~\cite{vm} as follows.

\begin{defn}
	The \textit{tensor product} of two signed graphs $\Sigma_1=(G_1,\sigma_1)$ and $\Sigma_2=(G_2,\sigma_2)$ is the signed graph $\Sigma=\Sigma_1\times \Sigma_2$ whose underlying graph is $G=G_1\times G_2$ and with the sign of an edge $(u_i,v_j)(u_k,v_l)$ of $G$ given by
	$$\sigma(((u_i,v_j)(u_k,v_l))=\sigma_1(u_1u_k)\sigma_2(v_jv_l).$$
\end{defn}
\begin{thm}\label{tpchar}\cite{ds1}
	Let $\Sigma_1$ and $\Sigma_2$ be two connected signed
	graphs of order at least 2. Then, the tensor product $\Sigma_1\times\Sigma_2$ is balanced if and only if $\Sigma_1$ and $\Sigma_2$ are both balanced or both antibalanced.
\end{thm}
\begin{thm}
	Let $\Sigma_1=(G_1,\sigma_1)$ and $\Sigma_2=(G_2,\sigma_2)$ be two signed graphs and $\Sigma_1\times\Sigma_2$ be their tensor product. Then
	\begin{itemize}
		\item [\rm{(i)}] $\bdim(\Sigma_1\times\Sigma_2) \leq \bdim(\Sigma_1)$, \mbox{  if $\Sigma_2$   is balanced}.
		\item [\rm{(ii)}] $\bdim(\Sigma_1\times\Sigma_2) \leq \bdim(\Sigma_2)$, \mbox{  if $\Sigma_1$   is balanced}
	\end{itemize}
\end{thm}
\begin{proof}
	Suppose $\bdim(\Sigma_1)=k$ and $\Sigma_2$ is balanced. Let $\zeta_1:V(\Sigma_1)\rightarrow\Omega^k$ and $\zeta_2:V(\Sigma_2)\rightarrow \{-1,+1\}$ be the corresponding switching functions. Then $\zeta: V(\Sigma_1\times\Sigma_2)\rightarrow\Omega^k$ defined by $\zeta((u_i,v_j))=\zeta_1(u_i)\zeta_2(v_j)$ for $1\leq i \leq |V(\Sigma_1)|$ and $1\leq j \leq |V(\Sigma_2)|$ switches $\Sigma_1\times\Sigma_2$ to all-positive, and hence $\bdim(\Sigma_1\times\Sigma_2)\leq k$.

Similar is the proof of the next part.
\end{proof}
\begin{rmk}
	Let $\Sigma_1=-K_3$ and $\Sigma_2=-K_2$ denote the all-negative signed complete graphs. Then by Theorem~\ref{tpchar}, $\bdim(\Sigma_1\times\Sigma_2)=1$. However, $\bdim(\Sigma_1)=\bar{\nu}(3)=3$. Hence, unlike the Cartesian product and the lexicographic products, there exist cases in which balancing dimension of the tensor product is strictly less than balancing dimension of its factor(s).

	As an example for the case where equality holds, consider $\Sigma_3=u_1u_2u_3$ and $\Sigma_4=v_1v_2v_3$ as the all-negative  and all-positive signed complete graphs on three vertices respectively. Then $(u_1,v_1)(u_2,v_2)(u_3,v_3)$ forms a negative triangle in $\Sigma_3\times\Sigma_4$. Thus, $\bdim(\Sigma_3\times\Sigma_4)=\bdim(\Sigma_4)$.
\end{rmk}
\subsection{Balancing dimension of the strong product}
Finally, we consider the strong product of signed graphs.
%

\begin{defn}\cite{bcd}
	The \textit{strong product} of two signed graphs $\Sigma_1=(G_1,\sigma_1)$ and $\Sigma_2=(G_2,\sigma_2)$ is the signed graph $\Sigma=\Sigma_1\boxtimes \Sigma_2$ whose underlying graph is $G=G_1\boxtimes G_2$ and with the sign of an edge $(u_i,v_j)(u_k,v_l)$ of $G$ given by
	$$\sigma((u_i,v_j)(u_k,v_l))=
\left\{
\begin{array}{ll}
	\sigma_1(u_iu_k)  & \mbox{if } u_i \sim u_k\  \mbox{and } v_j= v_l, \\

	\sigma_2(v_jv_l)  & \mbox{if } u_i=u_k\  \mbox{and } v_j\sim v_l, \\
	\sigma_1(u_iu_k)\sigma_2(v_jv_l)  & \mbox{if } u_i \sim u_k\  \mbox{and } v_j\sim v_l. \\
\end{array}
\right.$$
\end{defn}
\begin{lem}\label{strongswitcheq1}
Let $\Sigma_1$ and $\Sigma_2$ be two signed graphs and let $\Sigma_1\boxtimes\Sigma_2$ be their strong product.
\begin{itemize}
	\item [\rm{(i)}] If $\Sigma_1 \sim \Sigma_1'$, then $\Sigma_1\boxtimes\Sigma_2\sim\Sigma_1'\boxtimes\Sigma_2$.
	\item [\rm{(ii)}] If $\Sigma_2 \sim \Sigma_2'$, then $\Sigma_1\boxtimes\Sigma_2\sim\Sigma_1\boxtimes\Sigma_2'$.

\end{itemize}

\end{lem}
\begin{proof}
	Let $\sigma_1$, $\sigma_1'$, $\sigma_2$, $\sigma_2'$, $\sigma$,  $\sigma'$ and $\sigma''$ denote the signatures of $\Sigma_1$, $\Sigma_1'$, $\Sigma_2$, $\Sigma_2'$,  $\Sigma_1\boxtimes\Sigma_2$,  $\Sigma_1'\boxtimes\Sigma_2$, and  $\Sigma_1\boxtimes\Sigma_2'$ respectively.

	Since $\Sigma_1 \sim \Sigma_1'$, there exist a switching function $\eta:V(\Sigma_1)\rightarrow\{-1,1\}$ such that $\Sigma_1^\eta=\Sigma_1'$. Define the map $\eta':V(\Sigma_1\boxtimes\Sigma_2)\rightarrow\{-1,1\}$ as $\eta'((u_i,v_j))=\eta(u_i)$ for $1\leq i \leq |V(\Sigma_1)|$ and $1\leq j \leq |V(\Sigma_2)|$.
	 Then, for any edge $(u_i,v_j)(u_k,v_l)$ in $\Sigma_1\boxtimes\Sigma_2$, we have, 	$\sigma^{\eta'}((u_i,v_j)(u_k,v_l))=\sigma'((u_i,v_j)(u_k,v_l))$. Thus, $(\Sigma_1\boxtimes\Sigma_2)^{\eta'}=\Sigma_1'\boxtimes\Sigma_2$ and hence $\Sigma_1\boxtimes\Sigma_2\sim \Sigma_1'\boxtimes\Sigma_2$.

	 To prove $(ii)$, consider the map $\mu':V(\Sigma_1\boxtimes\Sigma_2)\rightarrow\{-1,1\}$ defined by $\mu'((u_i,v_j))=\mu(v_j)$ for $1\leq i \leq |V(\Sigma_1)|$ and $1\leq j \leq |V(\Sigma_2)|$, where  $\mu$ is the switching function that switches $\Sigma_2$ to $\Sigma_2'$.
\end{proof}

Using Lemma~\ref{strongswitcheq1} we arrive at the following theorem.
\begin{thm}\label{strongswitcheq}
	Let $\Sigma_1$ and $\Sigma_2$ be two signed graphs and let $\Sigma_1\boxtimes\Sigma_2$ be their strong product. If $\Sigma_1 \sim \Sigma_1'$ and $\Sigma_2 \sim \Sigma_2'$, then $\Sigma_1\boxtimes\Sigma_2\sim\Sigma_1'\boxtimes\Sigma_2'$.
\end{thm}
\begin{cor}
	If $\Sigma_1$ and $\Sigma_2$ are balanced, then so is $\Sigma_1\boxtimes\Sigma_2$.
\end{cor}
\begin{rmk}
	If $\Sigma_1$ and $\Sigma_2$ are antibalanced, it need not imply that their strong product $\Sigma_1\boxtimes\Sigma_2$ is antibalanced. As an example, consider  $\Sigma_1=\Sigma_2=-K_2$. Then their strong product is the balanced signed graph $\Sigma_1\boxtimes\Sigma_2=(K_4,\sigma)$. Hence the signed graph $-(\Sigma_1\boxtimes\Sigma_2)=(K_4,-\sigma)$ contains an unbalanced triangle, making it unbalanced. Thus, $\Sigma_1\boxtimes\Sigma_2$ is not antibalanced.
\end{rmk}
\begin{thm}
	Let $\Sigma_1=(G_1,\sigma_1)$ and $\Sigma_2=(G_2,\sigma_2)$ be two signed graphs and $\Sigma_1\boxtimes\Sigma_2$ be their strong product. Then
	$$\bdim(\Sigma_1\boxtimes\Sigma_2)=
	\begin{cases}
		\bdim(\Sigma_1), \mbox{  if $\Sigma_2$   is balanced}\\
		\bdim(\Sigma_2), \mbox{  if $\Sigma_1$   is balanced}.
	\end{cases}
	$$
\end{thm}

\begin{proof}
	Suppose $\bdim(\Sigma_1)=k$ and $\Sigma_2$ is balanced. Let $\zeta_1:V(\Sigma_1)\rightarrow\Omega^k$ and $\zeta_2:V(\Sigma_2)\rightarrow \{-1,1\}$ be the corresponding switching functions. Now, the function $\zeta: V(\Sigma_1\boxtimes\Sigma_2)\rightarrow\Omega^k$ defined by $\zeta((u_i,v_j))=\zeta_1(u_i)\zeta_2(v_j)$ switches $\Sigma_1\boxtimes\Sigma_2$ to all-positive, and hence $\bdim(\Sigma_1\boxtimes\Sigma_2)\leq k$. However, since $\Sigma_1$ is a subgraph of $\Sigma_1\boxtimes\Sigma_2$, we must have $\bdim(\Sigma_1\boxtimes\Sigma_2)=k=\bdim(\Sigma_1)$.

Similar is the proof of the next part.
\end{proof}
\section{Conclusion and Scope}
In this paper, we have studied the properties of balancing dimension of various products of signed graphs, namely, the Cartesian product, the lexicographic product, the tensor product and the strong product. We found the relationship between balancing dimensions of various signed graph products and their factors, provided one of them is balanced or all-positive.  We also computed the balancing dimensions of the Cartesian product of  unbalanced cycles, Cartesian product of antibalanced signed complete graphs and lexicographic product of antibalanced signed complete graphs. We also proved some results on switching equivalence in the case of the lexicographic products and the strong product. Finding the balancing dimensions of products of general unbalanced signed graphs, finding the relation between balancing dimensions of signed graph products and its factors,  and studying  properties of balancing dimension of other existing products of signed graphs are some exciting areas for further investigation.
\section*{Acknowledgments}
The first author would like to acknowledge his gratitude to the University Grants Commission (UGC), India, for providing financial support in the form of Junior Research fellowship (NTA Ref. No.: 191620039346). The authors express their sincere gratitude to Professor Thomas Zaslavsky, Binghamton University (SUNY), Binghamton, and Prof. Shahul Hameed K, K M M Government Women’s College, Kannur,   for their  valuable suggestions.


\end{document}